\documentclass[12pt]{article}

\usepackage{amsmath, amsthm, amssymb, stmaryrd, color, tikz}
\usetikzlibrary{calc,positioning}

\newtheorem{theorem}{Theorem}[section]

\newtheorem{lemma}[theorem]{Lemma}

\theoremstyle{definition}

\def\Longmapsto{\DOTSB\mapstochar\Longrightarrow}

\newcommand\pair[2]{\left\langle #1, #2 \right\rangle}
\newcommand\triple[3]{\left\langle #1, #2, #3 \right\rangle}

\newcommand{\<}{\left\langle}
\renewcommand{\>}{\right\rangle}

\begin{document}

\title{On the number of variables in undecidable superintuitionistic propositional calculi}

\author{Grigoriy V. Bokov\\
        \small Department of Mathematical Theory of Intelligent Systems\\
        \small Lomonosov Moscow State University\\
        \small Moscow, Russian Federation\\
        \small E-mail: bokov@intsys.msu.ru}

\maketitle

\begin{abstract}
In this paper, we construct an undecidable 3-variable superintuitionistic propositional calculus, i.e., a finitely axiomatizable extension of the intuitionistic propositional calculus with axioms containing only 3 variables. Since there are no 2-variable superintuitionistic propositional calculi, this is the minimal possible number of variables.
\end{abstract}

\maketitle

\section{Introduction}

Decidability is the important property of propositional calculi, it means that the set of their derivable formulas (or theorems) can be effectively determined. A natural question is how to separate classes of decidable and undecidable calculi. On the other hand, since undecidable propositional calculi can be used as a base for obtaining ``negative'' results to various algorithmic problems, it is of interest to find the simplest possible calculus of that class. There are many possible ways to separate decidable and undecidable calculi. A significant and simplest way is to describe the number of variables in their axioms.

In 1949, Linial and Post~\cite{LinialPost:49} found the first undecidable propositional calculus. In 1975, Hughes and Singletary~\cite{HughesSingletary:75:TPI} proved that there is an undecidable propositional calculus with axioms containing 3 variables. In 1976, Hughes~\cite{Hughes:76:TVIC} constructed an undecidable implicational propositional calculus using axioms in 2 variables. Finally, Gladstone in 1979~\cite{Gladstone:79:DOPC} proved that every 1-variable propositional calculus is decidable.

The first undecidable superintuitionistic propositional calculus was built in 1978 by Shehtman~\cite{Shehtman:78:USC, Shehtman:82:UPC}. Axioms of this calculus contain 7 variables. Later Chagrov in 1994~\cite{Chagrov:94:UPSL} did the same using axioms with only 4 variables. In~\cite[Sections 16.9]{Chagrov:97:ML} he noted that it is unknown whether there exist undecidable superintuitionistic propositional calculi with axioms in 2 or 3 variables.

In~\cite{Gladstone:70:NVA} Gladstone proved that the following formula
\begin{equation*}
  A = (p \to q) \to ((q \to r) \to (p \to r))
\end{equation*}
is not derivable from the set of all 2-variable tautologies by modus ponens and substitution. Since $A$ is an intuitionistic tautology, therefore a 2-variable propositional calculus cannot derive all intuitionistic tautologies. If we combine this with Gladstone's result for 1-variable propositional calculi, we get that there are no undecidable superintuitionistic propositional calculi with axioms containing less than 3 variables. The aim of this paper is to construct an undecidable 3-variable superintuitionistic propositional calculus.

This paper is organized as follows. In the next section we introduce the basic terminology and notation. In Section 3 we state and prove our main result. Finally, in Section 4 we give some concluding remarks and discuss further directions of research.

\section{Definitions}

In this section, we recall definitions of the intuitionistic propositional calculus and Kripke semantics. For more details we refer the reader to~\cite{Chagrov:97:ML}.

First, we introduce some notation. Let us consider the language consisting of an infinite set of propositional variables $\mathcal{V}$, brackets, and the signature \mbox{$\Sigma = \{\bot, \wedge, \vee, \to\}$}, where $\bot$ is the constant symbol, $\wedge$, $\vee$ and $\to$ are binary connectives. Letters $p, q, x, y$, etc., are used to denote propositional variables. We define $\neg$, $\leftrightarrow$ and $\top$ as the usual abbreviations: $\neg A := A \to \bot$, $A \leftrightarrow B = (A \to B) \wedge (B \to A)$, and $\top = \neg \bot$.

\emph{Propositional formulas} or $\Sigma$-\emph{formulas} are built up from the signature $\Sigma$, propositional variables from $\mathcal{V}$, and brackets in the usual way. For example, the following notations
\begin{equation*}
  x, \quad \neg A, \quad (A \wedge B), \quad (A \vee B), \quad (A \to B)
\end{equation*}
are formulas if $A$, $B$ are formulas. Capital letters $A, B, C$, etc., are used to denote propositional formulas. Throughout the paper, we omit some parentheses in formulas whenever it does not lead to confusion.

By a \emph{propositional calculus} or a $\Sigma$-\emph{calculus} we mean a finite set $P$ of $\Sigma$-formulas referred to as \emph{axioms} together with two rules of inference:

1) \emph{modus ponens}
\begin{equation*}
  A,~A \to B~\vdash~B,
\end{equation*}

2) \emph{substitution}
\begin{equation*}
  A~\vdash~\sigma A,
\end{equation*}
where $\sigma A$ is a substitution instance of $A$, i.e., the result of applying a substitution $\sigma$ to the formula $A$.

Denote by $[P]$ the set of derivable (or provable) formulas of a calculus $P$. A \emph{derivation} in $P$ is defined from the axioms and the rules of inference in the usual way. The statement that a formula $A$ is derivable from $P$ is denoted by $P \vdash A$.

Let us introduce the following pre-order relation on the set of all propositional calculi. We write $P_1 \leq P_2$ (or, equivalently, $P_2 \geq P_1$) if each derivable formula of $P_1$ is also derivable from $P_2$, i.e., if $[P_1] \subseteq [P_2]$. We write $P_1 \sim P_2$ and say that two calculi $P_1$ and $P_2$ are \emph{equivalent} if $[P_1] = [P_2]$. Finally, we write $P_1 < P_2$ if $[P_1] \subsetneq [P_2]$.

An \emph{intuitionistic Kripke frame} is a pair $\mathfrak{F} = \< W, R \>$ consisting of a nonempty set $W$ and a partial order $R$ on $W$, which is reflexive, transitive and antisymmetric, i.e., $\mathfrak{F}$ is just a partially ordered set. The elements of $W$ are called the \emph{points} (or \emph{worlds}) of the frame $\mathfrak{F}$, and the relation $R$ is called the \emph{accessibility relation}. If for some $w,w' \in W$ the relation $w R w'$ holds, we say that $w'$ is \emph{accessible} from $w$ or $w$ sees $w'$. We write $w \leq_R w'$ (or $w' \geq_R w$) iff $w R w'$.

A \emph{valuation} in an intuitionistic frame $\mathfrak{F} = \< W, R \>$ is a map $\mathfrak{V}$ associating with each propositional variable $p \in \mathcal{V}$ some (possibly empty) subset $\mathfrak{V}(p)$ of $W$ such that, for every $w \in \mathfrak{V}(p)$ and every $w' \in W$, $w \leq_R w'$ implies $w' \in \mathfrak{V}(p)$.

An \emph{intuitionistic Kripke model} is a pair $\mathfrak{M} = \< \mathfrak{F}, \mathfrak{V} \>$, where $\mathfrak{F}$ is an intuitionistic frame and $\mathfrak{V}$ is a valuation in $\mathfrak{F}$.

Let $\mathfrak{M} = \< \mathfrak{F}, \mathfrak{V} \>$ be an intuitionistic Kripke model and $w$ be a point in the frame $\mathfrak{F} = \< W, R \>$. By induction on the construction of a formula $A$ we define a relation $(\mathfrak{M}, w) \models A$, which is read as $A$ \emph{is true at} $w$ \emph{in} $\mathfrak{M}$:

\bigskip

\begin{tabular}{lcl}
  $(\mathfrak{M}, w) \not\models \bot$ & & \\
  $(\mathfrak{M}, w) \models p$ & $\Longleftrightarrow$ & $w \in \mathfrak{V}(p)$; \\
  $(\mathfrak{M}, w) \models A \wedge B$ & $\Longleftrightarrow$ & $(\mathfrak{M}, w) \models A$ and $(\mathfrak{M}, w) \models B$; \\
  $(\mathfrak{M}, w) \models A \vee B$ & $\Longleftrightarrow$ & $(\mathfrak{M}, w) \models A$ or $(\mathfrak{M}, w) \models B$; \\
  $(\mathfrak{M}, w) \models A \to B$ & $\Longleftrightarrow$ & for all $w' \in W$ such that $w \leq_R w'$, \\
  & & $(\mathfrak{M}, w') \models A$ implies  $(\mathfrak{M}, w') \models B$. \\
\end{tabular}

\medskip

\noindent From the definition it follows that

\medskip

\begin{tabular}{lcl}
  $(\mathfrak{M}, w) \models \top$ & & \\
  $(\mathfrak{M}, w) \models \neg A$ & $\Longleftrightarrow$ & for all $w' \in W$ such that $w \leq_R w'$, $(\mathfrak{M}, w') \not\models A$. \\
\end{tabular}

\bigskip

\noindent If $(\mathfrak{M}, w) \models A$ does not hold, i.e., $(\mathfrak{M}, w) \not\models A$, we say that $A$ is \emph{refuted at the point} $w$ \emph{in} $\mathfrak{M}$.

We say that $A$ is \emph{valid in a model} $\mathfrak{M} = \< \mathfrak{F}, \mathfrak{V} \>$ defined on a frame $\mathfrak{F} = \< W, R \>$ if $(\mathfrak{M}, w) \models A$ for all $w \in W$; if $A$ is valid in $\mathfrak{M}$, we write $\mathfrak{M} \models A$. We say that $A$ is \emph{valid in a frame} $\mathfrak{F} = \< W, R \>$ if $A$ is valid in every model based on $\mathfrak{F}$; if $A$ is valid in $\mathfrak{F}$, we write $\mathfrak{F} \models A$. We say that $A$ is \emph{true at a point} $w$ in a frame $\mathfrak{F}$ if $(\mathfrak{M}, w) \models A$ for every model $\mathfrak{M}$ defined on $\mathfrak{F}$; if $A$ is true at the point $w$ in frame $\mathfrak{F}$, we write $(\mathfrak{F}, w) \models A$. If $\mathfrak{M}$ is fixed we write $w \models A$ instead of $(\mathfrak{M}, w) \models A$.

We define the intuitionistic propositional calculus $\mathbf{Int}$ as the smallest propositional calculus containing the following set of axioms:
\begin{center}
  \begin{tabular}{ll}
    $(\to_1)$ & $p \to (q \to p)$ \\
    $(\to_2)$ & $(p \to (q \to r)) \to ((p \to q) \to (q \to r))$ \\
    $(\wedge_1)$ & $p \wedge q \to p$ \\
    $(\wedge_2)$ & $p \wedge q \to q$ \\
    $(\wedge_3)$ & $p \to (q \to p \wedge q)$ \\
    $(\vee_1)$ & $p \to p \vee q$ \\
    $(\vee_2)$ & $q \to p \vee q$ \\
    $(\vee_3)$ & $(p \to r) \to ((q \to r) \to (p \vee q \to r))$ \\
    $(\neg_1)$ & $(p \to q) \to ((p \to \neg q) \to \neg p)$ \\
    $(\neg_2)$ & $p \to (\neg p \to q)$ \\
  \end{tabular}
\end{center}
It is well known that
\begin{equation*}
  \mathbf{Int} \vdash A \quad \Longleftrightarrow \quad \mathfrak{F} \models A, \text{ for every Kripke frame } \mathfrak{F}.
\end{equation*}

By a \emph{superintuitionistic propositional calculus} we mean a propositional calculus obtained from $\mathbf{Int}$ by adding a finite set of new axioms. If $M$ is a finite set of propositional formulas, then a propositional calculus obtained from $\mathbf{Int}$ by adding new axioms $M$ is denoted by $\mathbf{Int} + M$. Since
\begin{equation*}
  \mathbf{Int} + \{A_1, \ldots, A_n\} \sim \mathbf{Int} + A_1 \wedge \ldots \wedge A_n,
\end{equation*}
we can assume that a superintuitionistic propositional calculus is a calculus $\mathbf{Int} + A$ for some intuitionistic propositional formula $A$.

\section{Main result}

Our main result is the following theorem.

\begin{theorem} \label{T:main}
There is a 3-variable intuitionistic propositional formula $A$ such that $\mathbf{Int} + A$ is undecidable.
\end{theorem}

First, we recall what a Minsky machine is and encode configurations of a Minsky machine by superintuitionistic propositional formulas. Next, we construct a Kripke model refuting all codes of derivable configurations. Finally, we encode instructions of a Minsky machine $\mathcal{M}$ by a single superintuitionistic formula $A_{\mathcal{M}}$ and formally reduce the configuration problem of $\mathcal{M}$ to the derivation problem of a superintuitionistic propositional calculus $\textbf{Int} + A_{\mathcal{M}}$.

\subsection{Minsky machine}

There are many algorithmic formalisms to prove the undecidability of a propositional calculus~\cite{Bokov:2015:UPPC}. For example, the undecidability of a calculus contained in the classical~\cite{Bokov:2009}, intuitionistic~\cite{Bokov:2015:URA} propositional calculus or in another subcalculus~\cite{Bokov:2015:UPPC} can be easily proved by using tag systems. But for extensions of the intuitionistic propositional calculus, this is very hard~\cite{Popov:81,Skvortsov:85}. For this reason, in order to prove the undecidability of superintuitionistic propositional calculi we will use an algorithmic formalism which is called \emph{Minsky machines}~\cite{Minsky:67:CFI}. In~\cite{Chagrov:97:ML} Chagrov mentioned that it is the most convenient formalism for being simulated by modal and intuitionistic formulas.

In accordance with~\cite{Chagrov:97:ML} we define a \emph{Minsky machine} as a finite set of instructions for transforming triples $\triple{s}{m}{n}$ of natural numbers, called \emph{configurations}, where $s$ is the number of the instruction to be executed at the next step (referred to as the \emph{current machine state}), and $m, n \in \mathbb{N}$~\footnote{We assume that $\mathbb{N} = \{0, 1, 2, \ldots \}$.}. Each instruction has one of the following four forms:
\begin{equation*}
  \begin{array}{ll}
  s \ \mapsto \ \triple{t}{1}{0}, & s \ \mapsto \ \triple{t}{-1}{0} / \triple{u}{0}{0}, \\
  s \ \mapsto \ \triple{t}{0}{1}, & s \ \mapsto \ \triple{t}{0}{-1} / \triple{u}{0}{0},
  \end{array}
\end{equation*}
where $s,t,u$ are the machine states. Note that all Minsky machines are assumed to be deterministic, i.e., they may not contain distinct instructions with the same numbers.

As an example, let us consider the applying of first two instructions. The instruction
\begin{equation*}
  s \ \mapsto \ \triple{t}{1}{0}
\end{equation*}
transforms $\triple{s}{m}{n}$ into $\triple{t}{m+1}{n}$, and the instruction
\begin{equation*}
  s \ \mapsto \ \triple{t}{-1}{0} / \triple{u}{0}{0}
\end{equation*}
transforms $\triple{s}{m}{n}$ into $\triple{t}{m-1}{n}$ if $m > 0$ and into $(u, m, n)$ if $m = 0$. The meaning of the others is defined analogously.

Let $\mathcal{M}$ be a Minsky machine, then the notation $\triple{s}{m}{n} \stackrel{\mathcal{M}}{\longmapsto} \triple{t}{k}{l}$ means that the configuration $\triple{t}{k}{l}$ is obtained from $\triple{s}{m}{n}$ by applying an instruction of machine $\mathcal{M}$ once. We write $\triple{s}{m}{n} \stackrel{\mathcal{M}}{\Longmapsto} \triple{t}{k}{l}$ if the configuration $\triple{t}{k}{l}$ is obtained from $\triple{s}{m}{n}$ by applying instructions of machine $\mathcal{M}$ in finitely many steps (possibly, in $0$ steps). Particularly, we always have $\triple{s}{m}{n} \stackrel{\mathcal{M}}{\Longmapsto} \triple{s}{m}{n}$.

The \emph{configuration problem} for a Minsky machine $M$ and a configuration $\triple{s}{m}{n}$ is, given a configuration $\triple{t}{k}{l}$, to determine whether $\triple{s}{m}{n} \stackrel{\mathcal{M}}{\Longmapsto} \triple{t}{k}{l}$.

\begin{theorem}[Minsky,~\cite{Minsky:67:CFI}] \label{T:Minsky}
There exist a Minsky machine $\mathcal{M}$ and a configuration $\triple{s}{m}{n}$ for which the configuration problem is undecidable.
\end{theorem}

Let $\mathcal{M}$ be a Minsky machine and $\triple{s_0}{m_0}{n_0}$ a configuration for which the configuration problem is undecidable.

\subsection{Encoding of configurations}

Let $p$, $q$ and $r$ be three distinct propositional variables. Now we define some propositional formulas using only variables $p$, $q$, $r$, which encode configurations of Minsky machines. Note that some basic ideas of defining these formulas was found in~\cite{Chagrov:97:ML} and~\cite{Rybakov:2006:CIVFL}.

First, let us define the following groups of propositional formulas constructed from variables $p$, $q$ and $r$. If
\begin{gather*}
  \begin{aligned}
    \qquad  S_{-2}[x] & \,=\, \neg x, \qquad
    & S_{-1}[x] & \,=\, T_{-2}[x] \to x, \\
    \qquad  T_{-2}[x] & \,=\, \neg \neg x, \qquad
    & T_{-1}[x] & \,=\, S_{-1}[x] \to S_{-2}[x] \vee T_{-2}[x],
  \end{aligned} \\
  \begin{aligned}
    S_i[x] & \,=\, T_{i-1}[x] \to S_{i-1}[x] \vee T_{i-2}[x], \\
    T_i[x] & \,=\, \quad S_i[x]   \to S_{i-1}[x] \vee T_{i-1}[x],
  \end{aligned}
\end{gather*}
for all $i \geq 0$, then we define

\emph{Groups $(A^0)$ and $(B^0)$}:
\begin{equation*}
  A_i^0 \,=\, S_{i+3}[r], \;\ B_i^0 \,=\, T_{i+3}[r] \;\ \text{ for all } i \geq -5.
\end{equation*}

\noindent Let $C_1 = A_0^0$ and $C_2 = B_0^0$, then

\emph{Groups $(A^1)$ and $(B^1)$}:
\begin{gather*}
  A_i^1 \,=\, S_{i+3}[p], \;\ B_i^1 \,=\, T_{i+3}[p] \;\ \text{ for } i \in \{-3, -4, -5\}, \\
  \begin{aligned}
    A_{-2}^1   & \,=\, B_{-3}^1 \to A_{-3}^1 \vee B_{-4}^1, \qquad
    & A_{-1}^1 & \,=\, B_{-2}^1 \to A_{-2}^1 \vee B_{-3}^1, \\
    B_{-2}^1   & \,=\, A_{-3}^1 \to      C_1 \vee B_{-3}^1, \qquad
    & B_{-1}^1 & \,=\, A_{-2}^1 \to A_{-3}^1 \vee B_{-2}^1,
  \end{aligned} \\
  \begin{aligned}
    A_i^1 & \,=\, C_2 \wedge B_{i-1}^1 \to C_1 \vee A_{i-1}^1 \vee B_{i-2}^1, \\
    B_i^1 & \,=\, C_2 \wedge A_{i-1}^1 \to C_1 \vee A_{i-2}^1 \vee B_{i-1}^1, \;\ \text{ for all } i \geq 0;
  \end{aligned}
\end{gather*}

\emph{Groups $(A^2)$ and $(B^2)$}:
\begin{gather*}
  A_i^2 \,=\, S_{i+3}[q], \;\ B_i^2 \,=\, T_{i+3}[q] \;\ \text{ for } i \in \{-3, -4, -5\}, \\
  \begin{aligned}
    A_{-2}^2   & \,=\, B_{-3}^2 \to A_{-3}^2 \vee B_{-4}^2, \qquad
    & A_{-1}^2 & \,=\, B_{-2}^2 \to A_{-2}^2 \vee B_{-3}^2, \\
    B_{-2}^2   & \,=\, A_{-3}^2 \to      C_2 \vee B_{-3}^2, \qquad
    & B_{-1}^2 & \,=\, A_{-2}^2 \to A_{-3}^2 \vee B_{-2}^2,
  \end{aligned} \\
  \begin{aligned}
    A_i^2 & \,=\, C_1 \wedge B_{i-1}^2 \to C_2 \vee A_{i-1}^2 \vee B_{i-2}^2, \\
    B_i^2 & \,=\, C_1 \wedge A_{i-1}^2 \to C_2 \vee A_{i-2}^2 \vee B_{i-1}^2, \;\ \text{ for all } i \geq 0.
  \end{aligned}
\end{gather*}

\noindent Note that the groups $(A^0)$, $(B^0)$ contain only variable $r$, $(A^1)$, $(B^1)$ contain only variables $r$, $p$, and $(A^2)$, $(B^2)$ contain only variables $r$, $q$. Now we define formulas encoding configurations of the Minsky machine $\mathcal{M}$.

\emph{Group $(E)$}:
\begin{multline*}
  E_{s,m,n} = A_{3s+2}^0 \wedge B_{3s+2}^0 \wedge A_{m+1}^1 \wedge B_{m+1}^1 \wedge A_{n+1}^2 \wedge B_{n+1}^2 \to \\
  \to A_{3s+1}^0 \vee B_{3s+1}^0 \vee A_m^1 \vee B_m^1 \vee A_n^2 \vee B_n^2,
\end{multline*}
for all $s,m,n \geq 0$. The formula $E_{s,m,n}$ is called the \emph{code} of a configuration $\triple{s}{m}{n}$.

Denote by $(A)$ and $(B)$ the following sets of formulas:
\begin{equation*}
  \begin{array}{lll}
    (A) & = & (A^0) \cup (A^1) \cup (A^2), \\
    (B) & = & (B^0) \cup (B^1) \cup (B^2), \\
  \end{array}
\end{equation*}
and by $M$ the set of formulas:
\begin{equation*}
  M = (A) \cup (B) \cup (E).
\end{equation*}

\subsection{Kripke model refuting codes of derivable configurations}

In this section, we construct a Kripke model $\mathfrak{M} = \< \mathfrak{F}, \mathfrak{V} \>$ refuting all formulas from $M$, i.e., for every formula from $M$, there exists a unique maximal point, at which this formula is refuted.

First, let us define the following equivalence relation $\sim_{\mathcal{M}}$ on the set of all configurations $\{ \triple{s}{m}{n} \mid s,m,n \geq 0\}$:
\begin{equation*}
  \triple{s}{m}{n} \sim_{\mathcal{M}} \triple{t}{k}{l} \leftrightharpoons \triple{s}{m}{n} \stackrel{\mathcal{M}}{\Longmapsto} \triple{t}{k}{l} \text{ and } \triple{t}{k}{l} \stackrel{\mathcal{M}}{\Longmapsto} \triple{s}{m}{n}.
\end{equation*}
Denote by $[s,m,n]$ the equivalence class of a configuration $\triple{s}{m}{n}$:
\begin{equation*}
  [s,m,n] = \{ \triple{t}{k}{l} \mid \triple{s}{m}{n} \sim_{\mathcal{M}} \triple{t}{k}{l} \}.
\end{equation*}
The set of all equivalence classes of relation $\sim_{\mathcal{M}}$ is denoted by $\mathcal{E}_{\mathcal{M}}$.

Let us define the relation $\stackrel{\mathcal{M}}{\Longmapsto}$ on the set of equivalence classes $\mathcal{E}_{\mathcal{M}}$:
\begin{equation*}
  [s,m,n] \stackrel{\mathcal{M}}{\Longmapsto} [t,k,l] \leftrightharpoons \triple{s}{m}{n} \stackrel{\mathcal{M}}{\Longmapsto} \triple{t}{k}{l}.
\end{equation*}
Greek letters $\alpha, \beta, \gamma$, etc., are used to denote equivalence classes. Denote by $\alpha_0$ the equivalence class of the initial configuration $\triple{s_0}{m_0}{n_0}$, i.e., $\alpha_0 = [s_0, m_0, n_0]$.

Now we define a Kripke frame $\mathfrak{F} = \< W, R \>$ as follows. Let
\begin{equation*}
      \bigcup\limits_{\substack{i \geq -5,\\ j \in \{0,1,2\}}} \{ a_i^j, b_i^j \} \cup
      \bigcup\limits_{\substack{\alpha \in \mathcal{E}_{\mathcal{M}}:\\ \alpha_0 \stackrel{\mathcal{M}}{\Longmapsto} \alpha}} \{e_{\alpha}\}.
\end{equation*}
To define the accessibility relation $R$ on $W$, we consider the following groups of relations:

\emph{Group $R_i^j$, $i \geq -4,\ j \in \{0,1,2\}$}:
\begin{align*}
  R_{-4}^j & = \left\{ \pair{a_{-4}^j}{a_{-5}^j},\ \pair{b_{-4}^j}{a_{-5}^j},\ \pair{b_{-4}^j}{b_{-5}^j} \right\}, \\
  R_{-3}^j & = \left\{ \pair{a_{-3}^j}{a_{-4}^j},\ \pair{a_{-3}^j}{b_{-5}^j},\ \pair{b_{-3}^j}{a_{-4}^j},\ \pair{b_{-3}^j}{b_{-4}^j} \right\}, \\
  R_i^0    & = \left\{ \pair{a_{i}^0}{a_{i-1}^0},\ \pair{a_{i}^0}{b_{i-2}^0},\ \pair{b_{i}^0}{a_{i-1}^0},\ \pair{b_{i}^0}{b_{i-1}^0} \right\} \;\ \text{ for all } i \geq -2 \text{ and } \\
  R_{-2}^1 & = \left\{ \pair{a_{-2}^1}{a_{-3}^1},\ \pair{a_{-2}^1}{b_{-4}^1},\ \pair{b_{-2}^1}{a_0^0},\ \pair{b_{-2}^1}{b_{-3}^1} \right\}, \\
  R_{-2}^2 & = \left\{ \pair{a_{-2}^2}{a_{-3}^2},\ \pair{a_{-2}^2}{b_{-4}^2},\ \pair{b_{-2}^2}{b_0^0},\ \pair{b_{-2}^2}{b_{-3}^2} \right\}, \\
  R_i^j    & = \left\{ \pair{a_i^j}{a_{i-1}^j},\ \pair{a_i^j}{b_{i-2}^j},\ \pair{b_i^j}{a_{i-2}^j},\ \pair{b_i^j}{b_{i-1}^j} \right\} \; \text{ for all } i \geq -1,\ j \in \{1,2\};
\end{align*}

\emph{Group $R_{s,m,n}$, $s,m,n \geq 0,\ \alpha_0 \stackrel{\mathcal{M}}{\Longmapsto} [s,m,n]$}:
\begin{align*}
  R_{s,m,n} = & \left\{ \pair{e_{[s,m,n]}}{a_{3s+1}^0}, \pair{e_{[s,m,n]}}{b_{3s+1}^0}, \pair{e_{[s,m,n]}}{a_m^1}, \right.\\
              & \;\; \left. \pair{e_{[s,m,n]}}{b_m^1}, \pair{e_{[s,m,n]}}{a_n^2}, \pair{e_{[s,m,n]}}{b_n^2} \right\}.
\end{align*}

\noindent Let
\begin{equation*}
  R' = \bigcup\limits_{\substack{i \geq -4,\\ j \in \{0,1,2\}}} R_i^j \cup
       \bigcup\limits_{\substack{s,m,n \geq 0:\\ \alpha_0 \stackrel{\mathcal{M}}{\Longmapsto} [s,m,n] }} R_{s,m,n} \cup
       \bigcup\limits_{\substack{\alpha, \beta \in \mathcal{E}_{\mathcal{M}}:\\ \alpha \stackrel{\mathcal{M}}{\Longmapsto} \beta}} \{ \pair{e_{\alpha}}{e_{\beta}} \}.
\end{equation*}
We take as $R$ the reflexive and transitive closure of $R'$.

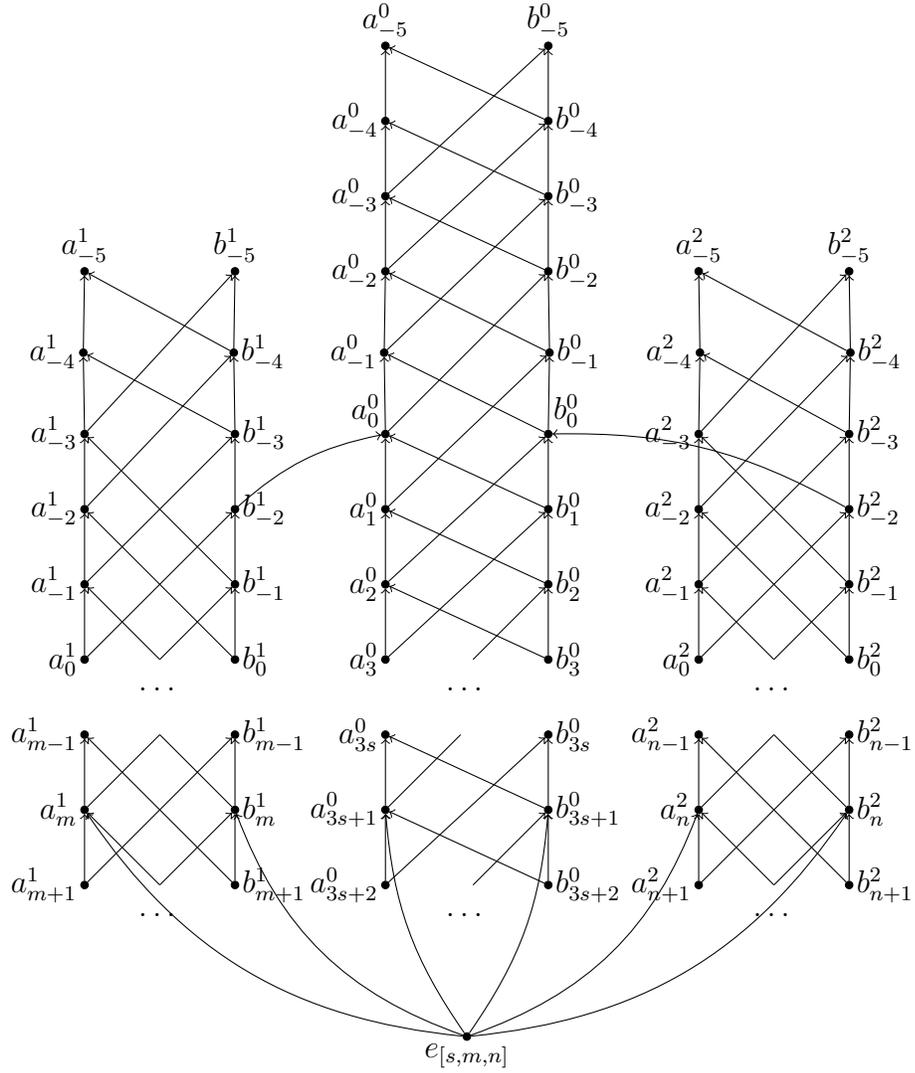
\begin{figure}

\centering

\begin{tikzpicture} [dot/.style={draw, circle, inner sep=1pt, fill}, tmp/.style={inner sep=1pt}, every label/.style={draw=none, inner sep=1pt, fill=none}]

\node[tmp] (top) {};
\node[tmp] (center) [below=4mm of top] {};

\node[dot] (a0m5) [label=90:  $a^0_{-5}$,  above left=4cm and 1cm of center] {};
\node[dot] (b0m5) [label=90:  $b^0_{-5}$,  above right=4cm and 1cm of center] {};
\node[dot] (a0m4) [label=180: $a^0_{-4}$,  above left=3cm and 1cm of center] {};
\node[dot] (b0m4) [label=0:   $b^0_{-4}$,  above right=3cm and 1cm of center] {};
\node[dot] (a0m3) [label=180: $a^0_{-3}$,  above left=2cm and 1cm of center] {};
\node[dot] (b0m3) [label=0:   $b^0_{-3}$,  above right=2cm and 1cm of center] {};
\node[dot] (a0m2) [label=180: $a^0_{-2}$,  above left=1cm and 1cm of center] {};
\node[dot] (b0m2) [label=0:   $b^0_{-2}$,  above right=1cm and 1cm of center] {};
\node[dot] (a0m1) [label=180: $a^0_{-1}$,  left=1cm of center] {};
\node[dot] (b0m1) [label=0:   $b^0_{-1}$,  right=1cm of center] {};
\node[dot] (a00)  [label=120: $a^0_{0}$,   below left=1cm and 1cm of center] {};
\node[dot] (b00)  [label=60:  $b^0_{0}$,   below right=1cm and 1cm of center] {};
\node[dot] (a01)  [label=180: $a^0_{1}$,   below left=2cm and 1cm of center] {};
\node[dot] (b01)  [label=0:   $b^0_{1}$,   below right=2cm and 1cm of center] {};
\node[dot] (a02)  [label=180: $a^0_{2}$,   below left=3cm and 1cm of center] {};
\node[dot] (b02)  [label=0:   $b^0_{2}$,   below right=3cm and 1cm of center] {};
\node[dot] (a03)  [label=180: $a^0_{3}$,   below left=4cm and 1cm of center] {};
\node[dot] (b03)  [label=0:   $b^0_{3}$,   below right=4cm and 1cm of center] {};
\node[tmp] (dot01)[label=90:  $\dots$,     below=45mm of center] {};
\node[dot] (a0sp) [label=180: $a^0_{3s}$, below left=5cm and 1cm of center] {};
\node[dot] (b0sp) [label=0:   $b^0_{3s}$, below right=5cm and 1cm of center] {};
\node[dot] (a0s)  [label=180: $a^0_{3s+1}$,   below left=6cm and 1cm of center] {};
\node[dot] (b0s)  [label=0:   $b^0_{3s+1}$,   below right=6cm and 1cm of center] {};
\node[dot] (a0ss) [label=180: $a^0_{3s+2}$, below left=7cm and 1cm of center] {};
\node[dot] (b0ss) [label=0:   $b^0_{3s+2}$, below right=7cm and 1cm of center] {};
\node[tmp] (dot02)[label=90:  $\dots$,     below=75mm of center] {};

\node[dot] (a1m5) [label=90:  $a^1_{-5}$,  above left=1cm and 5cm of center] {};
\node[dot] (b1m5) [label=90:  $b^1_{-5}$,  above left=1cm and 3cm of center] {};
\node[dot] (a1m4) [label=180: $a^1_{-4}$,  left=5cm of center] {};
\node[dot] (b1m4) [label=0:   $b^1_{-4}$,  left=3cm of center] {};
\node[dot] (a1m3) [label=180: $a^1_{-3}$,  below left=1cm and 5cm of center] {};
\node[dot] (b1m3) [label=0:   $b^1_{-3}$,  below left=1cm and 3cm of center] {};
\node[dot] (a1m2) [label=180: $a^1_{-2}$,  below left=2cm and 5cm of center] {};
\node[dot] (b1m2) [label=0:   $b^1_{-2}$,  below left=2cm and 3cm of center] {};
\node[dot] (a1m1) [label=180: $a^1_{-1}$,  below left=3cm and 5cm of center] {};
\node[dot] (b1m1) [label=0:   $b^1_{-1}$,  below left=3cm and 3cm of center] {};
\node[dot] (a10)  [label=180: $a^1_{0}$,   below left=4cm and 5cm of center] {};
\node[dot] (b10)  [label=0:   $b^1_{0}$,   below left=4cm and 3cm of center] {};
\node[tmp] (dot11)[label=90:  $\dots$,     below left=45mm and 4cm of center] {};
\node[dot] (a1mp) [label=180: $a^1_{m-1}$, below left=5cm and 5cm of center] {};
\node[dot] (b1mp) [label=0:   $b^1_{m-1}$, below left=5cm and 3cm of center] {};
\node[dot] (a1m)  [label=180: $a^1_{m}$,   below left=6cm and 5cm of center] {};
\node[dot] (b1m)  [label=0:   $b^1_{m}$,   below left=6cm and 3cm of center] {};
\node[dot] (a1ms) [label=180: $a^1_{m+1}$, below left=7cm and 5cm of center] {};
\node[dot] (b1ms) [label=0:   $b^1_{m+1}$, below left=7cm and 3cm of center] {};
\node[tmp] (dot12)[label=90:  $\dots$,     below left=75mm and 4cm of center] {};

\node[dot] (a2m5) [label=90:  $a^2_{-5}$,  above right=1cm and 3cm of center] {};
\node[dot] (b2m5) [label=90:  $b^2_{-5}$,  above right=1cm and 5cm of center] {};
\node[dot] (a2m4) [label=180: $a^2_{-4}$,  right=3cm of center] {};
\node[dot] (b2m4) [label=0:   $b^2_{-4}$,  right=5cm of center] {};
\node[dot] (a2m3) [label=180: $a^2_{-3}$,  below right=1cm and 3cm of center] {};
\node[dot] (b2m3) [label=0:   $b^2_{-3}$,  below right=1cm and 5cm of center] {};
\node[dot] (a2m2) [label=180: $a^2_{-2}$,  below right=2cm and 3cm of center] {};
\node[dot] (b2m2) [label=0:   $b^2_{-2}$,  below right=2cm and 5cm of center] {};
\node[dot] (a2m1) [label=180: $a^2_{-1}$,  below right=3cm and 3cm of center] {};
\node[dot] (b2m1) [label=0:   $b^2_{-1}$,  below right=3cm and 5cm of center] {};
\node[dot] (a20)  [label=180: $a^2_{0}$,   below right=4cm and 3cm of center] {};
\node[dot] (b20)  [label=0:   $b^2_{0}$,   below right=4cm and 5cm of center] {};
\node[tmp] (dot21)[label=90:  $\dots$,     below right=45mm and 4cm of center] {};
\node[dot] (a2np) [label=180: $a^2_{n-1}$, below right=5cm and 3cm of center] {};
\node[dot] (b2np) [label=0:   $b^2_{n-1}$, below right=5cm and 5cm of center] {};
\node[dot] (a2n)  [label=180: $a^2_{n}$,   below right=6cm and 3cm of center] {};
\node[dot] (b2n)  [label=0:   $b^2_{n}$,   below right=6cm and 5cm of center] {};
\node[dot] (a2ns) [label=180: $a^2_{n+1}$, below right=7cm and 3cm of center] {};
\node[dot] (b2ns) [label=0:   $b^2_{n+1}$, below right=7cm and 5cm of center] {};
\node[tmp] (dot22)[label=90:  $\dots$,     below right=75mm and 4cm of center] {};

\node[dot] (e)    [label=270: $e_{[s,m,n]}$,     below=9cm of center] {};

\draw [->] (a0m4) edge (a0m5);
\draw [->] (b0m4) edge (a0m5) edge (b0m5);
\draw [->] (a0m3)  edge (a0m4) edge (b0m5);
\draw [->] (b0m3)  edge (a0m4) edge (b0m4);
\draw [->] (a0m2)  edge (a0m3) edge (b0m4);
\draw [->] (b0m2)  edge (a0m3) edge (b0m3);
\draw [->] (a0m1)  edge (a0m2) edge (b0m3);
\draw [->] (b0m1)  edge (a0m2) edge (b0m2);
\draw [->] (a00)  edge (a0m1) edge (b0m2);
\draw [->] (b00)  edge (a0m1) edge (b0m1);
\draw [->] (a01)  edge (a00) edge (b0m1);
\draw [->] (b01)  edge (a00) edge (b00);
\draw [->] (a02)  edge (a01) edge (b00);
\draw [->] (b02)  edge (a01) edge (b01);
\draw [<-] (b02) -- ($(b02) + (-1,-1)$);
\draw [->] (a03)  edge (a02) edge (b01);
\draw [->] (b03)  edge (a02) edge (b02);
\draw [->] (a0s)  edge (a0sp);
\draw [->] (b0s)  edge (a0sp) edge (b0sp);
\draw [-]  (a0s) -- ($(a0s) + (1,1)$);
\draw [<-] (b0s) -- ($(b0s) + (-1,-1)$);
\draw [->] (a0ss) edge (a0s) edge (b0sp);
\draw [->] (b0ss) edge (a0s) edge (b0s);

\draw [->] (a1m4) edge (a1m5);
\draw [->] (b1m4) edge (a1m5) edge (b1m5);
\draw [->] (a1m3) edge (a1m4) edge (b1m5);
\draw [->] (b1m3) edge (a1m4) edge (b1m4);
\draw [->] (a1m2) edge (a1m3) edge (b1m4);
\draw [->] (b1m2) edge [bend left=15] (a00) edge (b1m3);
\draw [->] (a1m1) edge (a1m2) edge (b1m3);
\draw [->] (b1m1) edge (a1m3) edge (b1m2);
\draw [<-] (a1m1) -- ($(a1m1) + (1,-1)$);
\draw [<-] (b1m1) -- ($(b1m1) + (-1,-1)$);
\draw [->] (a10)  edge (a1m1) edge (b1m2);
\draw [->] (b10)  edge (a1m2) edge (b1m1);
\draw [->] (a1m)  edge (a1mp);
\draw [->] (b1m)  edge (b1mp);
\draw [-]  (a1m) -- ($(a1m) + (1,1)$);
\draw [-]  (b1m) -- ($(b1m) + (-1,1)$);
\draw [<-] (a1m) -- ($(a1m) + (1,-1)$);
\draw [<-] (b1m) -- ($(b1m) + (-1,-1)$);
\draw [->] (a1ms) edge (a1m)  edge (b1mp);
\draw [->] (b1ms) edge (a1mp) edge (b1m);

\draw [->] (a2m4) edge (a2m5);
\draw [->] (b2m4) edge (a2m5) edge (b2m5);
\draw [->] (a2m3) edge (a2m4) edge (b2m5);
\draw [->] (b2m3) edge (a2m4) edge (b2m4);
\draw [->] (a2m2) edge (a2m3) edge (b2m4);
\draw [->] (b2m2) edge [bend right=15] (b00) edge (b2m3);
\draw [->] (a2m1) edge (a2m2) edge (b2m3);
\draw [->] (b2m1) edge (a2m3) edge (b2m2);
\draw [<-] (a2m1) -- ($(a2m1) + (1,-1)$);
\draw [<-] (b2m1) -- ($(b2m1) + (-1,-1)$);
\draw [->] (a20)  edge (a2m1) edge (b2m2);
\draw [->] (b20)  edge (a2m2) edge (b2m1);
\draw [->] (a2n)  edge (a2np);
\draw [->] (b2n)  edge (b2np);
\draw [-]  (a2n) -- ($(a2n) + (1,1)$);
\draw [-]  (b2n) -- ($(b2n) + (-1,1)$);
\draw [<-] (a2n) -- ($(a2n) + (1,-1)$);
\draw [<-] (b2n) -- ($(b2n) + (-1,-1)$);
\draw [->] (a2ns) edge (a2n)  edge (b2np);
\draw [->] (b2ns) edge (a2np) edge (b2n);

\draw [->] (e) edge [bend left=25] (a1m) edge [bend left=25] (b1m) edge [bend left=15] (a0s) edge [bend right=15] (b0s) edge [bend right=25] (a2n) edge [bend right=25] (b2n);

\end{tikzpicture}

\caption{Kripke model $\mathfrak{M}$.} \label{model}

\end{figure}

Let us define a valuation $\mathfrak{V}$ of the Kripke model $\mathfrak{M} = \< \mathfrak{F}, \mathfrak{V} \>$ in the following way:
\begin{equation*}
\begin{aligned}
  (\mathfrak{M}, w) \not\models r \quad \Longleftrightarrow \quad & w \leq_R a_{-4}^0 & \text{or } &\ w \leq_R b_{-5}^0; \\
  (\mathfrak{M}, w) \not\models p \quad \Longleftrightarrow \quad & w \leq_R a_{-4}^1 & \text{or } &\ w \leq_R b_{-5}^1; \\
  (\mathfrak{M}, w) \not\models q \quad \Longleftrightarrow \quad & w \leq_R a_{-4}^2 & \text{or } &\ w \leq_R b_{-5}^2.
\end{aligned}
\end{equation*}
The model $\mathfrak{M}$ is depicted on Figure~\ref{model}. Now we prove some basic semantic properties of the Kripke model $\mathfrak{M}$.

\begin{lemma} \label{L:Semantic}
Let $w$ be a world of $\mathfrak{M}$, then
\begin{align*}
  w \not\models A_i^j \quad \Longleftrightarrow \quad & w \leq_R a_i^j, \\
  w \not\models B_i^j \quad \Longleftrightarrow \quad & w \leq_R b_i^j
\end{align*}
for all $i \geq -4$ and $j \in \{0,1,2\}$.
\end{lemma}
\begin{proof}
By induction on $i \geq -4$.

\noindent
\textbf{Induction base} consists of the following cases:

1) $i = -4$. Let $x_0 = r$, $x_1 = p$, and $x_2 = q$.

Since $w \not\models x_j$ iff $w \leq_R a_{-4}^j$ or $w \leq_R b_{-5}^j$, we have that $A_{-4}^j$ is refuted at $a_{-4}^j$ and $B_{-4}^j$ is refuted at $b_{-4}^j$. Therefore, $w \not\models A_{-4}^j$ if $w \leq_R a_{-4}^j$ and $w \not\models B_{-4}^j$ if $w \leq_R b_{-4}^j$.

If $w \not\models A_{-4}^j$, then there exists a point $w' \geq_R w$ such that $w' \models \neg \neg x_j$ and $w' \not\models x_j$. By definition of the valuation $\mathfrak{V}$, we have either $w' \leq_R a_{-4}^j$ or $w' \leq_R b_{-5}^j$. Since $w' \models \neg \neg x_j$, therefore for all point $w'' \geq_R w'$ there is a point $w''' \geq_R w''$ such that $w''' \models x_j$. It is clear that $w' \nleq_R b_{-5}^j$. Hence, $w \leq_R a_{-4}^j$.

If $w \not\models B_{-4}^j$, then there exist points $w' \geq_R w$ and $w'' \geq_R w$ such that $w' \models x_j$ and $w'' \models \neg x_j$. By definition of the valuation $\mathfrak{V}$, we have $w' \nleq_R a_{-4}^j$, $w' \nleq_R b_{-5}^j$, and $w'' = b_{-5}^j$. If $w' \leq_R a_{-5}^{j'}$ or $w' \leq_R b_{-5}^{j'}$ for some $j' \in \{0,1,2\} \setminus \{j\}$, then $w \leq_R e_{[s,m,n]}$ for some $s,m,n \geq 0$ and therefore $w \leq_R b_{-4}^j$. Otherwise, $w' = a_{-5}^j$. Note that there is a unique point $w''' \geq_R w$ such that $w''' \leq_R a_{-5}^j$, $w''' \leq_R b_{-5}^j$, and $w''' \models A_{-4}^j$. It is easily seen that $w''' = b_{-4}^j$ and therefore $w \leq_R b_{-4}^j$.

2) $i = -3$.

Note that $w \not\models A_{-3}^j$ if $w \leq_R a_{-4}^j$, $w \leq_R b_{-5}^j$, $w \nleq_R b_{-4}^j$ and $w \not\models B_{-3}^j$ if $w \leq_R a_{-4}^j$, $w \leq_R b_{-4}^j$, $w \nleq_R a_{-3}^j$. Since $a_{-3}^j$ and $b_{-3}^j$ are unique maximal points satisfying this condition, we have that $w \not\models A_{-3}^j$ if $w \leq_R a_{-3}^j$ and $w \not\models B_{-3}^j$ if $w \leq_R b_{-3}^j$.

If $w \not\models A_{-3}^j$, then there exists a point $w' \geq_R w$ such that $w' \not\models A_{-4}^j,\ \neg \neg x_j$ and $w' \models B_{-4}^j$. So, $w' \leq_R a_{-4}^j$ and $w' \nleq_R b_{-4}^j$. Since $w' \not\models \neg \neg x_j$, there is a point $w'' \geq_R w'$ such that $w'' \models \neg x_j$. It is clear that $w'' = b_{-5}^j$. Evidently, the model $\mathfrak{M}$ contains only one point $a_{-3}^j$ satisfying the following condition: $w' \leq_R a_{-4}^j$, $w' \leq_R b_{-5}^j$ and $w' \nleq_R b_{-4}^j$. Hence, $w \leq_R a_{-3}^j$.

If $w \not\models B_{-3}^j$, then there exists a point $w' \geq_R w$ such that $w' \not\models A_{-4}^j,\ B_{-4}^j$ and $w' \models A_{-3}^j$. Then $w' \leq_R a_{-4}^j$, $w' \leq_R b_{-4}^j$ and $w' \nleq_R a_{-3}^j$. Evidently, the model $\mathfrak{M}$ contains only one point $b_{-3}^j$ satisfying this condition. Therefore, $w \leq_R b_{-3}^j$.

3) $i = -2$ and $j \in \{1,2\}$.

We have that $w \not\models A_{-2}^j$ if $w \leq_R a_{-3}^j$, $w \leq_R b_{-4}^j$, $w \nleq_R b_{-3}^j$ and $w \not\models B_{-2}^j$ if $w \leq_R c$, $w \leq_R b_{-3}^j$, $w \nleq_R a_{-2}^j$, where $c = a_0^0$ for $j = 1$ and $c = b_0^0$ for $j = 2$. Since $a_{-2}^j$ and $b_{-2}^j$ are unique maximal points satisfying this condition, we have that $w \not\models A_{-2}^j$ if $w \leq_R a_{-2}^j$ and $w \not\models B_{-2}^j$ if $w \leq_R b_{-2}^j$.

If $w \not\models A_{-2}^j$, then there exists a point $w' \geq_R w$ such that $w' \not\models A_{-3}^j,\ B_{-4}^j$ and $w' \models B_{-3}^j$. So, $w' \leq_R a_{-3}^j$, $w' \leq_R b_{-4}^j$, and $w' \nleq_R b_{-3}^j$. It is clear that the model $\mathfrak{M}$ contains only one point $a_{-2}^j$ satisfying this condition. Hence, $w \leq_R a_{-2}^j$.

If $w \not\models B_{-2}^j$, then there exists a point $w' \geq_R w$ such that $w' \not\models C_j,\ B_{-3}^j$ and $w' \models A_{-3}^j$. Then $w' \leq_R c$, $w' \leq_R b_{-3}^j$ and $w' \nleq_R a_{-3}^j$, where $c = a_0^0$ if $j = 1$ and $c = b_0^0$ if $j = 2$. Evidently, the model $\mathfrak{M}$ contains only one point $b_{-2}^j$ satisfying this condition. Therefore, $w \leq_R b_{-2}^j$.

4) $i = -1$ and $j \in \{1,2\}$. This case easily follows by analogy.

\noindent
\textbf{Induction step:} assume that $i \geq -2$ if $j = 0$ and $i \geq 0$ if $j \in \{1,2\}$. Without loss of generality, we can consider the case $j = 1$. The cases $j=0$ and $j=2$ are proved by analogy.

By induction assumption, we have that $w \not\models A_i^1$ if $w \leq_R a_0^0, a_{i-1}^1, b_{i-2}^1$, $w \nleq_R b_0^0, b_{i-1}^1$ and $w \not\models B_i^1$ if $w \leq_R a_0^0, a_{i-2}^1, b_{i-1}^1$, $w \nleq_R b_0^0, a_{i-1}^1$. Since $a_i^1$ and $b_i^1$ are unique maximal points satisfying this condition, we have that $w \not\models A_i^1$ if $w \leq_R a_i^1$ and $w \not\models B_i^1$ if $w \leq_R b_i^1$.

If $w \not\models A_i^1$, then there exists a point $w' \geq_R w$ such that $w' \not\models C_1,\ A_{i-1}^1,\ B_{i-2}^1$ and $w' \models C_2, B_{i-1}^1$. By induction hypothesis, we obtain that $w' \leq_R a_{i-1}^1$, $w' \leq_R b_{i-2}^1$, and $w' \nleq_R b_{i-1}^1$. So, $w' = a_i^1$ and $w \leq_R a_i^1$ by definition of the accessibility relation $R$. Analogously, if $w \not\models B_i^1$, then $w \leq_R b_i^1$. The lemma is proved.
\end{proof}

\begin{lemma}
Let $w$ be a world of $\mathfrak{M}$, then
\begin{equation*}
  w \not\models E_{s,m,n} \quad \Longleftrightarrow \quad w \leq_R e_{[s,m,n]}
\end{equation*}
for all $s,m,n \geq 0$ such that $\alpha_0 \stackrel{\mathcal{M}}{\Longmapsto} [s,m,n]$.
\end{lemma}

The proof is trivial by definition of the accessibility relation $R$. Finally, we prove the key lemma of this section.

\begin{lemma} \label{L:Semantic:2}
If $(\mathfrak{F},w) \not\models E_{s,m,n}$, then $e_{[s,m,n]} \in W$ and $w \leq_R e_{[s,m,n]}$ for all $s,m,n \geq 0$.
\end{lemma}
\begin{proof}
Let $\mathfrak{M}' = \<\mathfrak{F},\mathfrak{V}'\>$ be a Kripke model such that $(\mathfrak{M}',w) \not\models E_{s,m,n}$. Since $w \not\models E_{s,m,n}$, there is a point $w' \geq_R w$ such that the formulas $A_{3s+1}^0$, $B_{3s+1}^0$, $A_m^1$, $B_m^1$, $A_n^2$, $B_n^2$ are refuted at $w'$, and the formulas $A_{3s+2}^0$, $B_{3s+2}^0$, $A_{m+1}^1$, $B_{m+1}^1$, $A_{n+1}^2$, $B_{n+1}^2$ are true at $w'$.

Denote by $w_s^a$ and $w_s^b$ points of the frame $\mathfrak{F}$ such that
\begin{enumerate}
  \item $w_s^a \geq_R w'$, $w_s^a \models B_{3s}^0$, and $w_s^a \not\models A_{3s}^0, B_{3s-1}^0$;
  \item $w_s^b \geq_R w'$, $w_s^b \models A_{3s+1}^0$, and $w_s^b \not\models A_{3s}^0, B_{3s}^0$.
\end{enumerate}
It is clear that these points exist.

If $w_s^a$ or $w_s^b$ are in $\{a_n^0, b_n^0\}$ for some $n \geq -5$, then $a_{-5}^0 \in \mathfrak{V}'(r)$, $a_{-4}^0, b_{-5}^0 \notin \mathfrak{V}'(r)$. By analogy with Lemma~\ref{L:Semantic}, it is not hard to prove that $w_s^a = a_{3s+1}^0$, $w_s^b = b_{3s+1}^0$ by induction on $s \geq 0$.

Let $w_s^a$ and $w_s^b$ are not in $\{a_n^0, b_n^0\}$ for all $n \geq -5$. Then there are $n \geq -5$ and $j \in \{1,2\}$ such that $w_s^a$ or $w_s^b$ are in $\{a_n^j, b_n^j\}$. Evidently, either $a_{-5}^0 \notin \mathfrak{V}'(r)$, $a_{-4}^0 \in \mathfrak{V}'(r)$, or $b_{-5}^0 \notin \mathfrak{V}'(r)$. We need to consider the following cases:
\begin{enumerate}
  \item $a_{-5}^j \notin \mathfrak{V}'(r)$. In this case, $A_{-4}^0$ is refuted at $a_0^0$ if $j = 1$ and $b_0^0$ if $j = 2$. Then it can easily be seen that $a_{-5}^0, b_{-5}^0 \in \mathfrak{V}'(r)$ and $a_0^0 \notin \mathfrak{V}'(r)$. If $b_{-5}^j \notin \mathfrak{V}'(r)$, then $B_{-4}^0$ is true at $w_s^a$, $w_s^b$, which is impossible. If $b_{-5}^j \in \mathfrak{V}'(r)$, then $B_{-4}^0$ is refuted at $a_{-3}^j$, $b_{-4}^j$ and therefore $A_{-3}^0$ is true at $w_s^a$, $w_s^b$, which is impossible. Hence $a_{-5}^j \in \mathfrak{V}'(r)$.
  \item $b_{-5}^j \in \mathfrak{V}'(r)$. In this case, $B_{-4}^0$ is refuted at $b_{-2}^j$. Then $A_{-3}^0$ is true at $w_s^a$, $w_s^b$, which is impossible. Hence $b_{-5}^j \notin \mathfrak{V}'(r)$.
  \item $a_{-4}^j \in \mathfrak{V}'(r)$. In this case, $A_{-4}^0$ is refuted at $a_0^0$ if $j = 1$ and $b_0^0$ if $j = 2$. As the above, we have that $a_{-5}^0, b_{-5}^0 \in \mathfrak{V}'(r)$ and $a_0^0 \notin \mathfrak{V}'(r)$. Then $B_{-4}^0$ is refuted at $a_{-3}^j$, $b_{-4}^j$ and therefore $A_{-3}^0$ is true at $w_s^a$, $w_s^b$, which is impossible. Hence $a_{-4}^j \notin \mathfrak{V}'(r)$.
\end{enumerate}
So, we have that $a_{-5}^j \in \mathfrak{V}'(r)$ and $a_{-4}^j, b_{-5}^j \notin \mathfrak{V}'(r)$. Then $A_{-4}^0$ is refuted at $a_{-4}^j$ and $B_{-4}^0$ is refuted at $b_{-4}^j$. It can easily be proved by induction on $s \geq 0$ that $w_s^a = a_{4s+2}^j$ and $w_s^b = b_{4s+3}^j$. Therefore, if $w_s^a = a_{4s+2}^{j_1}$ and $w_s^b = b_{4s+3}^{j_2}$ for some $j_1, j_2 \in \{1, 2\}$, then $w' \leq_R a_{4s+3}^{j_2}$ and therefore $A_{3s+2}^0$ is refuted at $w'$, which is impossible. Hence $w_s^a$ or $w_s^b$ are in $\{a_n^0, b_n^0\}$ for some $n \geq -5$ and therefore $a_{-5}^0 \in \mathfrak{V}'(r)$, $a_{-4}^0, b_{-5}^0 \notin \mathfrak{V}'(r)$.

Since $C_1$ is refuted at $w_1$ if $w_1 \leq_R a_0^0$ and $C_2$ is refuted at $w_2$ if $w_2 \leq_R b_0^0$, we have that, for a given $i \geq 0$ and $j \in \{1,2\}$, the formulas $A_i^j$, $B_i^j$ are refuted at $a_k^j$, $b_k^j$ for some $k \geq 0$ and true at $a_k^{j'}$, $b_k^{j'}$ for all $k \geq 0$, $j' \in \{0,1,2\} \setminus \{j\}$ by analogy with the above. Therefore, $a_{-5}^1 \in \mathfrak{V}'(p)$, $a_{-4}^1, b_{-5}^1 \notin \mathfrak{V}'(p)$ and $a_{-5}^2 \in \mathfrak{V}'(q)$, $a_{-4}^2, b_{-5}^2 \notin \mathfrak{V}'(q)$.

Now if we recall the proof of Lemma~\ref{L:Semantic}, then we obtain that $w' \leq_R a_{3s+1}^0$, $w' \leq_R b_{3s+1}^0$, $w' \leq_R a_m^1$, $w' \leq_R b_m^1$, $w' \leq_R a_n^2$, $w' \leq_R b_n^2$ and $w' \nleq_R a_{3s+2}^0$, $w' \nleq_R b_{3s+2}^0$, $w' \nleq_R a_{m+1}^1$, $w' \nleq_R b_{m+1}^1$, $w' \nleq_R a_{n+1}^2$, $w' \nleq_R b_{n+1}^2$. Evidently, the frame $\mathfrak{F}$ contains a unique maximal point $e_{[s,m,n]}$ satisfying this condition. Hence $e_{[s,m,n]} \in W$ and $w \leq_R w' \leq_R e_{[s,m,n]}$. The lemma is proved.
\end{proof}

\subsection{Key formulas}

In this section, we consider the key formulas depending on variables $p$, $q$, $r$. First, let us define the following formulas $F_k = F_k[p,q,x,y]$ and $G_k = G_k[p,q,x,y]$ in variables $p$, $q$, $x$ and $y$:
\begin{equation*}
  \begin{array}{lll}
    F_0 & = & p, \\
    G_0 & = & q, \\
    F_1 & = & y \wedge q \to x \vee p, \\
    G_1 & = & y \wedge p \to x \vee q, \text{ and } \\
    F_k & = & y \wedge G_{k-1} \to x \vee F_{k-1} \vee G_{k-2}, \\
    G_k & = & y \wedge F_{k-1} \to x \vee G_{k-1} \vee F_{k-2}, \text{ for all } k \geq 2.
  \end{array}
\end{equation*}
Now we introduce the following key formulas:
\begin{equation*}
  \begin{array}{lll}
    F_k^1[p,q] & = & F_k[p,q,C_1,C_2], \\
    G_k^1[p,q] & = & G_k[p,q,C_1,C_2], \\
    F_k^2[p,q] & = & F_k[p,q,C_2,C_1], \\
    G_k^2[p,q] & = & G_k[p,q,C_2,C_1].
  \end{array}
\end{equation*}
Note that the formulas $F_k^m$ and $G_k^m$ are depending on three variables $p$, $q$, and $r$, for all $k \geq 0$ and $m \in \{1,2\}$.

Besides, we define the following auxiliary formulas:
\begin{equation*}
  \begin{array}{lll}
    P_{i,j} & = & (C_2 \to C_1 \vee A_{i}^1 \vee B_{i-1}^1) \wedge (C_1 \to C_2 \vee A_{i}^2 \vee B_{i-1}^2), \\
    Q_{i,j} & = & (C_2 \to C_1 \vee A_{i-1}^1 \vee B_{i}^1) \wedge (C_1 \to C_2 \vee A_{i-1}^2 \vee B_{i}^2), \\
  \end{array}
\end{equation*}
for all $i,j \geq -1$. The following lemma is describing the basic properties of the key formulas.

\begin{lemma} \label{L:KeyFormulas:3}
For all $i, j \geq -1$, $k \geq 1$ and $m \in \{1,2\}$,
\begin{equation*}
  \begin{array}{lll}
  \textbf{Int} & \vdash & F_k^m[P_{i,j}, Q_{i,j}] \leftrightarrow A_{n+k}^m, \\
  \textbf{Int} & \vdash & G_k^m[P_{i,j}, Q_{i,j}] \leftrightarrow B_{n+k}^m,
  \end{array}
\end{equation*}
where
\begin{equation*}
  n =
  \left\{
    \begin{array}{ll}
      i, & m = 1; \\
      j, & m = 2.
    \end{array}
  \right.
\end{equation*}
\end{lemma}
\begin{proof}
By induction on $k \geq 1$. Without loss of generality, we can assume that $m = 1$. The basis of induction consists of two cases: $k = 1$ and $k = 2$.

\textbf{Induction base:} $k = 1$. In this case we have
\begin{equation*}
  \begin{array}{rcl}
  F_1^1[P_{i,j}, Q_{i,j}] &=& C_2 \wedge Q_{i,j} \to C_1 \vee P_{i,j}, \\
  G_1^1[P_{i,j}, Q_{i,j}] &=& C_2 \wedge P_{i,j} \to C_1 \vee Q_{i,j}.
  \end{array}
\end{equation*}
It can easily be checked that the following derivations holds in $\textbf{Int}$:
\begin{equation*}
  \begin{array}{rclrcl}
  \textbf{Int} & \vdash & C_2 \wedge B_i^1 \to C_2 \wedge Q_{i,j}, &
  \textbf{Int} & \vdash & C_1 \vee P_{i,j} \to (C_2 \to C_1 \vee A_i^1 \vee B_{i-1}^1), \\
  \textbf{Int} & \vdash & C_2 \wedge A_i^1 \to C_2 \wedge P_{i,j}, &
  \textbf{Int} & \vdash & C_1 \vee Q_{i,j} \to (C_2 \to C_1 \vee A_{i-1}^1 \vee B_i^1).
  \end{array}
\end{equation*}
Hence,
\begin{equation*}
  \begin{array}{lll}
  \textbf{Int} & \vdash & F_1^1[P_{i,j}, Q_{i,j}] \to A_{i+1}^1, \\
  \textbf{Int} & \vdash & G_1^1[P_{i,j}, Q_{i,j}] \to B_{i+1}^1.
  \end{array}
\end{equation*}

Conversely, since the formulas $A_{i-1}^1 \to A_i^1$ and $B_{i-1}^1 \to B_i^1$ are derivable from $\textbf{Int}$, we have
\begin{equation*}
  \begin{array}{lll}
  \textbf{Int},\ A_{i+1}^1 & \vdash & C_1 \vee A_{i-1}^1 \vee B_i^1 \to (C_2 \to C_1 \vee A_i^1 \vee B_{i-1}^1), \\
  \textbf{Int},\ B_{i+1}^1 & \vdash & C_1 \vee A_i^1 \vee B_{i-1}^1 \to (C_2 \to C_1 \vee A_{i-1}^1 \vee B_i^1)
  \end{array}
\end{equation*}
and therefore the following derivations holds in $\textbf{Int}$:
\begin{equation*}
  \begin{array}{lll}
  \textbf{Int},\ A_{i+1}^1 & \vdash & C_2 \wedge (C_2 \to C_1 \vee A_{i-1}^1 \vee B_i^1) \to P_{i,j}, \\
  \textbf{Int},\ B_{i+1}^1 & \vdash & C_2 \wedge (C_2 \to C_1 \vee A_i^1 \vee B_{i-1}^1) \to Q_{i,j}.
  \end{array}
\end{equation*}
Hence,
\begin{equation*}
  \begin{array}{lll}
  \textbf{Int} & \vdash & A_{i+1}^1 \to F_1^1[P_{i,j}, Q_{i,j}], \\
  \textbf{Int} & \vdash & B_{i+1}^1 \to G_1^1[P_{i,j}, Q_{i,j}].
  \end{array}
\end{equation*}

\textbf{Induction base:} $k = 2$. In this case we have
\begin{equation*}
  \begin{array}{rcl}
  \textbf{Int} & \vdash & F_2^1[P_{i,j}, Q_{i,j}] \leftrightarrow (C_2 \wedge B_{i+1}^1 \to C_1 \vee A_{i+1}^1 \vee Q_{i,j}), \\
  \textbf{Int} & \vdash & G_2^1[P_{i,j}, Q_{i,j}] \leftrightarrow (C_2 \wedge A_{i+1}^1 \to C_1 \vee B_{i+1}^1 \vee P_{i,j}).
  \end{array}
\end{equation*}
Furthermore, it follows easily that:
\begin{equation*}
  \begin{array}{rclrcl}
  \textbf{Int} & \vdash & C_2 \wedge B_i^1 \to Q_{i,j}, &
  \textbf{Int} & \vdash & Q_{i,j} \to (C_2 \to C_1 \vee A_{i+1}^1 \vee B_i^1), \\
  \textbf{Int} & \vdash & C_2 \wedge A_i^1 \to P_{i,j}, &
  \textbf{Int} & \vdash & P_{i,j} \to (C_2 \to C_1 \vee A_i^1 \vee B_{i+1}^1).
  \end{array}
\end{equation*}
Hence,
\begin{equation*}
  \begin{array}{lll}
  \textbf{Int} & \vdash & F_2^1[P_{i,j}, Q_{i,j}] \leftrightarrow A_{i+2}^1, \\
  \textbf{Int} & \vdash & G_2^1[P_{i,j}, Q_{i,j}] \leftrightarrow B_{i+2}^1.
  \end{array}
\end{equation*}

\textbf{Induction step} is straightforward and left to the reader. The lemma is proved.
\end{proof}

\subsection{Encoding of the Minsky machine}

Now we encode instructions of the Minsky machine $\mathcal{M}$ as superintuitionistic formulas such that derivations from $\textbf{Int}$ and these formulas are simulate transformations of $\mathcal{M}$.

First, let us define the following formulas containing only tree variables $p$, $q$, $r$:
\begin{equation*}
  \begin{aligned}
    \hat{E}_{s,i,j} & = A_{3s+2}^0 \wedge B_{3s+2}^0 \wedge F_{i+1}^1 \wedge G_{i+1}^1 \wedge F_{j+1}^2 \wedge G_{j+1}^2 \to\\
                    & \to A_{3s+1}^0 \vee B_{3s+1}^0 \vee F_{i}^1 \vee G_{i}^1 \vee F_{j}^2 \vee G_{j}^2, \\
    \hat{E}_{s,0,*} & = A_{3s+2}^0 \wedge B_{3s+2}^0 \wedge A_{1}^1 \wedge B_{1}^1 \to
                        A_{3s+1}^0 \vee B_{3s+1}^0 \vee A_0^1 \vee B_0^1 \vee q, \\
    \hat{E}_{s,*,0} & = A_{3s+2}^0 \wedge B_{3s+2}^0 \wedge A_{1}^2 \wedge B_{1}^2 \to
                        A_{3s+1}^0 \vee B_{3s+1}^0 \vee p \vee A_0^2 \vee B_0^2, \\
    \hat{E}_{s,0,0} & = E_{s,0,0},
  \end{aligned}
\end{equation*}
where $s \geq 0$, $i,j \geq 1$. By Lemma~\ref{L:KeyFormulas:3}, we have the following evident lemma.

\begin{lemma} \label{L:Equivalence}
For all $s,m,n \geq 0$,
\begin{equation*}
  \textbf{Int} \vdash E_{s,m,n} \leftrightarrow
  \left\{
    \begin{array}{ll}
      \hat{E}_{s,i,j}[P_{m-i,n-j}, Q_{m-i,n-j}], & 1 \leq i \leq m+1,\\ & 1 \leq j \leq n+1; \\
      A_{n+1}^2 \wedge B_{n+1}^2 \to \hat{E}_{s,0,*}[p, A_n^2 \vee B_n^2], & m = 0,\ n \geq 1; \\
      A_{m+1}^1 \wedge B_{m+1}^1 \to \hat{E}_{s,*,0}[A_m^1 \vee B_m^1, q], & m \geq 1,\ n = 0; \\
      \hat{E}_{s,0,0}, & m = 0,\ n = 0.
    \end{array}
  \right.
\end{equation*}
\end{lemma}

Let
\begin{equation*}
  \varphi(x) =
  \left\{
    \begin{array}{ll}
      x-1, & x \geq 1; \\
      0,   & x = 0; \\
      0,   & x = *.
    \end{array}
  \right.
\end{equation*}
Now we prove that if the Kripke frame $\mathfrak{F}$ refutes $\hat{E}_{s,i,j}$ then it refutes $\hat{E}_{s,i,j}$ at a point $e_{[s,m,n]}$ for some $m \geq \varphi(i)$, $n \geq \varphi(j)$ such that $\alpha_0 \stackrel{\mathcal{M}}{\Longmapsto} [s,m,n]$.

\begin{lemma} \label{L:Semantic:3}
If $(\mathfrak{F},w) \not\models \hat{E}_{s,i,j}$, then $w \leq_R e_{[s,m,n]}$ for some $m \geq \varphi(i)$, $n \geq \varphi(j)$ such that $\alpha_0 \stackrel{\mathcal{M}}{\Longmapsto} [s,m,n]$.
\end{lemma}
\begin{proof}
Let $\mathfrak{M}' = \<\mathfrak{F},\mathfrak{V}'\>$ be a Kripke model such that $(\mathfrak{M}',w) \not\models \hat{E}_{s,i,j}$. Since $w \not\models \hat{E}_{s,i,j}$, there is a point $w' \geq_R w$ such that the formulas $A_{3s+1}^0$ and $B_{3s+1}^0$ are refuted at $w'$, and the formulas $A_{3s+2}^0$ and $B_{3s+2}^0$ are true at $w'$. By the proof of Lemma~\ref{L:Semantic:2}, we have that $w \leq_R w' \leq_R e_{[s,m,n]}$ for some $m, n \geq 0$ such that $\alpha_0 \stackrel{\mathcal{M}}{\Longmapsto} [s,m,n]$. It is clear that $m = 0$ if $i = 0$ and $n = 0$ if $j = 0$. Hence, in order to prove the lemma it is sufficient to show that $m \geq i-1$, $n \geq j-1$ for some $i \geq 1$, $j \geq 1$.

If $i \geq 1$, then the formulas $F_{i}^1$, $G_{i}^1$ are refuted at $w'$ and the formulas $F_{i+1}^1$, $G_{i+1}^1$ are true at $w'$. Now we prove that if $F_k^1$ is refuted at a point $f_k^1$ and $G_k^1$ is refuted at a point $g_k^1$, then $f_k^1 \leq_R c_{k+l-1}^1$ and $g_k^1 \leq_R d_{k+l-1}^1$ for some $l \geq 0$ and $\{c,d\} = \{a,b\}$. By induction on $k \geq 1$.

\textbf{Induction base:} $k = 1$. In this case, there are points $w_f \geq_R f_1^1$ and $w_g \geq_R g_1^1$ such that
\begin{enumerate}
  \item $C_1$ is refuted at $w_f$, $w_g$, then the Kripke frame $\mathfrak{F}$ contains pathes of length 5 from $w_f$, $w_g$ to maximal points and therefore $w_f \leq_R c_0^{j_1}$ and $w_g \leq_R d_0^{j_2}$ for some $j_1, j_2 \in \{0,1,2\}$ and $c,d \in \{a,b\}$;
  \item $C_2$ is true at $w_f$, $w_g$, therefore $w_f \nleq_R b_{-1}^0$, $w_g \nleq_R b_{-1}^0$ by the proof of  Lemma~\ref{L:Semantic:2};
  \item $w_f \in \mathfrak{V}'(q) \setminus \mathfrak{V}'(p)$ and $w_g \in \mathfrak{V}'(p) \setminus \mathfrak{V}'(q)$, therefore $w_f$, $w_g$ are incomparable points.
\end{enumerate}
Thus, $w_f = c_{i'}^1$ and $w_g = d_{j'}^1$ for some $i',j' \geq 0$ such that $|i' - j'| < 2$, and $\{c,d\} = \{a,b\}$.

\textbf{Induction base:} $k = 2$. In this case, there are points $w_f \geq_R f_2^1$ and $w_g \geq_R g_2^1$ such that
\begin{enumerate}
  \item $F_1^1$ is refuted at $w_f$ and $G_1^1$ is refuted at $w_g$, therefore $w_f \leq_R c_{i'}^1$, $w_g \leq_R d_{j'}^1$;
  \item $F_1^1$ is true at $w_g$ and $G_1^1$ is true at $w_f$, therefore $w_f \nleq_R d_{j'}^1$, $w_g \nleq_R c_{i'}^1$;
  \item $w_f, w_g \notin \mathfrak{V}'(p) \cup \mathfrak{V}'(q)$, therefore $w_f \neq c_{i'}^1$, $w_g \neq d_{j'}^1$.
\end{enumerate}
Thus, $w_f = c_{i''}^1$, $w_g = d_{j''}^1$ and $(w_f, d_{j'}^1)$, $(c_{i'}^1, w_g)$ and $(w_f, w_g)$ are pairs of incomparable points. So, we have
\begin{equation*}
  \begin{aligned}
    i' & < & i'' & < & j' + 2, \\
    j' & < & j'' & < & i' + 2.
  \end{aligned}
\end{equation*}
Since $|i' - j'| < 2$, it can easily be checked that $i' = j' = l$ and $i'' = j'' = l + 1$ for some $l \geq 0$.

\textbf{Induction step:} $k > 2$. Let the induction assumption be satisfied for all $2 \leq k' < k$, then there are points $w_f \geq_R f_k^1$ and $w_g \geq_R g_k^1$ such that
\begin{enumerate}
  \item $F_{k-1}^1$, $G_{k-2}^1$ are refuted at $w_f$, therefore $w_f \leq_R c_{k+l-2}^1$, $w_f \leq_R d_{k+l-3}^1$;
  \item $G_{k-1}^1$, $F_{k-2}^1$ are refuted at $w_g$, therefore $w_g \leq_R d_{k+l-2}^1$, $w_g \leq_R c_{k+l-3}^1$;
  \item $G_{k-1}^1$ is true at $w_f$ and $F_{k-1}^1$ is true at $w_g$, therefore $w_f \nleq_R d_{k+l-2}^1$ and $w_g \nleq_R c_{k+l-2}^1$.
\end{enumerate}
Thus, $w_f = c_{k+l-1}^1$ and $w_g = d_{k+l-1}^1$.

Since $F_{i}^1$, $G_{i}^1$ are refuted at $w'$ and the formulas $F_{i+1}^1$, $G_{i+1}^1$ are true at $w'$, we have $w' \leq_R c_{i+l-1}^1$, $w' \leq_R d_{i+l-1}^1$ and $w' \nleq_R c_{i+l}^1$, $w' \nleq_R d_{i+l}^1$. Therefore, $m = i + l - 1 \geq i-1$.

If $j \geq 1$, then the proof are similar. Hence, $n \geq j-1$. The lemma is proved.
\end{proof}

\noindent Next, we define the formula $Ax(I)$ simulating the instruction $I$ of the Minsky machine $\mathcal{M}$:
\begin{enumerate}
  \item If $I$ is an instruction of the form $s \ \mapsto \ \triple{t}{1}{0}$, then $Ax(I)$ is the following formula
  \begin{equation*}
    \hat{E}_{t,2,1} \to \hat{E}_{s,1,1};
  \end{equation*}

  \item If $I$ is $s \ \mapsto \ \triple{t}{0}{1}$, then $Ax(I)$ is
  \begin{equation*}
    \hat{E}_{t,1,2} \to \hat{E}_{s,1,1};
  \end{equation*}

  \item If $I$ is $s \ \mapsto \ \triple{t}{-1}{0} / \triple{u}{0}{0}$, then $Ax(I)$ is
  \begin{equation*}
    (\hat{E}_{t,1,1} \to \hat{E}_{s,2,1}) \wedge (\hat{E}_{u,0,*} \to \hat{E}_{s,0,*});
  \end{equation*}

  \item If $I$ is $s \ \mapsto \ \triple{t}{0}{-1} / \triple{u}{0}{0}$, then $Ax(I)$ is
  \begin{equation*}
    (\hat{E}_{t,1,1} \to \hat{E}_{s,1,2}) \wedge (\hat{E}_{u,*,0} \to \hat{E}_{s,*,0}),
  \end{equation*}
\end{enumerate}
and the formula $Ax(\mathcal{M})$ simulating the behavior of $\mathcal{M}$ itself:
\begin{equation*}
  Ax(\mathcal{M}) = \bigwedge\limits_{I \in \mathcal{M}} Ax(I).
\end{equation*}

\begin{lemma} \label{L:Axiom}
$\mathfrak{F} \models Ax(\mathcal{M})$.
\end{lemma}
\begin{proof}
In order to prove the lemma it is sufficient to show that
\begin{equation*}
  \mathfrak{F} \models Ax(I)
\end{equation*}
for each instruction $I$. We need to consider the following 4 cases.

\textbf{Case 1:} $I$ is an instruction of the form $s \ \mapsto \ \triple{t}{1}{0}$, i.e.,
\begin{equation*}
  Ax(I) = \hat{E}_{t,2,1} \to \hat{E}_{s,1,1}.
\end{equation*}
If $(\mathfrak{F},w) \not\models Ax(I)$, then there is a Kripke model $\mathfrak{M}' = \<\mathfrak{F},\mathfrak{V}'\>$ such that $(\mathfrak{M}',w) \models \hat{E}_{t,2,1}$ and $(\mathfrak{M}',w) \not\models \hat{E}_{s,1,1}$. By Lemma~\ref{L:Semantic:3}, $w \leq_R e_{[s,m,n]}$ for some $m \geq 0$ and $n \geq 0$ such that $\alpha_0 \stackrel{\mathcal{M}}{\Longmapsto} [s,m,n]$. If we recall the proofs of Lemmas~\ref{L:Semantic:2} and~\ref{L:Semantic:3}, we obtain that the following statements hold in $\mathfrak{M}'$
\begin{enumerate}
  \item $A_{3t+1}^0$, $B_{3t+1}^0$ are refuted at $a_{3t+1}^0$, $b_{3t+1}^0$ and $A_{3t+2}^0$, $B_{3t+2}^0$ are true at them;
  \item $F_2^1$, $G_2^1$ are refuted at $c_{m+1}^1$, $d_{m+1}^1$ and $F_3^1$, $G_3^1$ are true at them, where $\{c, d\} = \{a,b\}$;
  \item $F_1^2$, $G_1^2$ are refuted at $c_n^2$, $d_n^2$ and $F_2^2$, $G_2^2$ are true at them, where $\{c, d\} = \{a,b\}$.
\end{enumerate}
Since
\begin{equation*}
  \triple{s}{m}{n} \stackrel{\mathcal{M}}{\longmapsto} \triple{t}{m+1}{n},
\end{equation*}
we have that $e_{[t,m+1,n]} \in W$ and $e_{[s,m,n]} \leq_R e_{[t,m+1,n]}$. Hence $\hat{E}_{t,2,1}$ is refuted at $w$, which contradicts to that $(\mathfrak{M}',w) \models \hat{E}_{t,2,1}$. Therefore, $(\mathfrak{F},w) \models Ax(I)$.

\textbf{Case 2:} $I$ is an instruction of the form $s \ \mapsto \ \triple{t}{0}{1}$. The proof is analogous.

\textbf{Case 3:} $I$ is an instruction of the form $s \ \mapsto \ \triple{t}{-1}{0} / \triple{u}{0}{0}$, i.e.,
\begin{equation*}
  (\hat{E}_{t,1,1} \to \hat{E}_{s,2,1}) \wedge (\hat{E}_{u,0,*} \to \hat{E}_{s,0,*}).
\end{equation*}
Let $(\mathfrak{F},w) \not\models Ax(I)$. Then there is a Kripke model $\mathfrak{M}' = \<\mathfrak{F},\mathfrak{V}'\>$ such that
\begin{equation*}
  (\mathfrak{M}',w) \not\models (\hat{E}_{t,1,1} \to \hat{E}_{s,2,1}), (\hat{E}_{u,0,*} \to \hat{E}_{s,0,*}).
\end{equation*}
It is clear that if $(\mathfrak{M}',w) \not\models \hat{E}_{s,2,1}$, then $(\mathfrak{M}',w) \not\models \hat{E}_{t,1,1}$. Let $(\mathfrak{M}',w) \not\models \hat{E}_{s,0,*}$ for some point $w \in W$, then by Lemma~\ref{L:Semantic:3} $w \leq_R e_{[s,0,n]}$ for some $n \geq 0$ such that $\alpha_0 \stackrel{\mathcal{M}}{\Longmapsto} [s,0,n]$. If we recall the proofs of Lemmas~\ref{L:Semantic:2} and~\ref{L:Semantic:3} again, we obtain that
\begin{enumerate}
  \item $A_{3u+1}^0$, $B_{3u+1}^0$ are refuted at $a_{3u+1}^0$, $b_{3u+1}^0$ and $A_{3u+2}^0$, $B_{3u+2}^0$ are true at them;
  \item $A_0^1$, $B_0^1$ are refuted at $a_0^1$, $b_0^1$ and $A_1^1$, $B_1^1$ are true at them.
\end{enumerate}
Since
\begin{equation*}
  \triple{s}{0}{n} \stackrel{\mathcal{M}}{\longmapsto} \triple{u}{0}{n},
\end{equation*}
we have that $e_{[u,0,n]} \in W$ and $e_{[s,0,n]} \leq_R e_{[u,0,n]}$. Hence $(\mathfrak{M}',w) \not\models \hat{E}_{u,0,*}$ and therefore $(\mathfrak{F},w) \models Ax(I)$.

\textbf{Case 4:} $I$ is an instruction of the form $s \ \mapsto \ \triple{t}{0}{-1} / \triple{u}{0}{0}$. The proof is similar.

Thus, $\mathfrak{F} \models Ax(I)$ for each instruction $I \in \mathcal{M}$. The lemma is proved.
\end{proof}

\subsection{Reduction of configuration problem}

In this section we formally reduce the configuration problem of the Minsky machine $\mathcal{M}$ to the derivation problem of the superintuitionistic propositional calculus $\textbf{Int} + Ax(\mathcal{M})$.

\begin{lemma}
$\textbf{Int} + Ax(\mathcal{M}) \vdash E_{t,k,l} \to E_{s_0,m_0,n_0}$ iff $\triple{s_0}{m_0}{n_0} \stackrel{\mathcal{M}}{\Longmapsto} \triple{t}{k}{l}$.
\end{lemma}
\begin{proof}
If $\textbf{Int} + Ax(\mathcal{M}) \vdash E_{t,k,l} \to E_{s_0,m_0,n_0}$, then
\begin{equation*}
  \mathfrak{F} \models E_{t,k,l} \to E_{s_0,m_0,n_0}
\end{equation*}
by Lemma~\ref{L:Axiom}. If we recall that $E_{s_0,m_0,n_0}$ is refuted at $e_{[s_0,m_0,n_0]}$, then we obtain that $E_{t,k,l}$ is also refuted at $e_{[s_0,m_0,n_0]}$. By Lemma~\ref{L:Semantic:2}, $e_{[t,k,l]} \in  W$ and
\begin{equation*}
  e_{[s_0,m_0,n_0]} \leq_R e_{[t,k,l]}.
\end{equation*}
Therefore, $\triple{s_0}{m_0}{n_0} \stackrel{\mathcal{M}}{\Longmapsto} \triple{t}{k}{l}$ by definition of Kripke frame $\mathfrak{F}$.

Conversely, if $\triple{s_0}{m_0}{n_0} \stackrel{\mathcal{M}}{\Longmapsto} \triple{t}{k}{l}$, then there exists a finite sequence $\triple{s_i}{m_i}{n_i}$, $0 \leq i \leq \mu$, such that $\triple{s_{\mu}}{m_{\mu}}{n_{\mu}} = \triple{t}{k}{l}$ and
\begin{equation*}
  \triple{s_{i}}{m_i}{n_i} \stackrel{\mathcal{M}}{\longmapsto} \triple{s_{i+1}}{m_{i+1}}{n_{i+1}}
\end{equation*}
for all $i$, $0 \leq i < \mu$. Let $\triple{s_{i+1}}{m_{i+1}}{n_{i+1}}$ be a result of applying of an instruction $I \in \mathcal{M}$.  We need to consider the following 4 cases.

\textbf{Case 1:} $I$ is an instruction of the form $s \ \mapsto \ \triple{t}{1}{0}$. Then $m_{i+1} = m_i + 1$ and $n_{i+1} = n_i$. By Lemma~\ref{L:Equivalence}, we have
\begin{equation*}
  \begin{array}{rcl}
    \textbf{Int} & \vdash & E_{s_{i+1},m_{i+1},n_{i+1}} \leftrightarrow \hat{E}_{s_{i+1},2,1}[P_{m_i-1,n_i-1}, Q_{m_i-1,n_i-1}], \\
    \textbf{Int} & \vdash & E_{s_i,m_i,n_i} \leftrightarrow \hat{E}_{s_i,1,1}[P_{m_i-1,n_i-1}, Q_{m_i-1,n_i-1}].
  \end{array}
\end{equation*}
Therefore $\textbf{Int} + Ax(\mathcal{M}) \vdash E_{s_{i+1},m_{i+1},n_{i+1}} \to E_{s_i,m_i,n_i}$

\textbf{Case 2:} $I$ is an instruction of the form $s \ \mapsto \ \triple{t}{0}{1}$. The proof is analogous.

\textbf{Case 3:} $I$ is an instruction of the form $s \ \mapsto \ \triple{t}{-1}{0} / \triple{u}{0}{0}$. If $m_{i+1} = m_i - 1 \geq 0$ and $n_{i+1} = n_i$. By Lemma~\ref{L:Equivalence}, we have
\begin{equation*}
  \begin{array}{rcl}
    \textbf{Int} & \vdash & E_{s_{i+1},m_{i+1},n_{i+1}} \leftrightarrow \hat{E}_{s_{i+1},1,1}[P_{m_i-2,n_i-1}, Q_{m_i-2,n_i-1}], \\
    \textbf{Int} & \vdash & E_{s_i,m_i,n_i} \leftrightarrow \hat{E}_{s_i,2,1}[P_{m_i-2,n_i-1}, Q_{m_i-2,n_i-1}].
  \end{array}
\end{equation*}
If $m_{i+1} = m_i = 0$ and $n_{i+1} = n_i$. By Lemma~\ref{L:Equivalence}, we have
\begin{equation*}
  \begin{array}{rcl}
    \textbf{Int} & \vdash & E_{s_{i+1},m_{i+1},n_{i+1}} \leftrightarrow \left( A_{n_i+1}^2 \wedge B_{n_i+1}^2 \to \hat{E}_{s_{i+1},0,*}[p, A_{n_i}^2 \vee B_{n_i}^2] \right), \\
    \textbf{Int} & \vdash & E_{s_i,m_i,n_i}             \leftrightarrow \left( A_{n_i+1}^2 \wedge B_{n_i+1}^2 \to \hat{E}_{s_{i},0,*}[p, A_{n_i}^2 \vee B_{n_i}^2] \right).
  \end{array}
\end{equation*}
Therefore $\textbf{Int} + Ax(\mathcal{M}) \vdash E_{s_{i+1},m_{i+1},n_{i+1}} \to E_{s_i,m_i,n_i}$.

\textbf{Case 4:} $I$ is an instruction of the form $s \ \mapsto \ \triple{t}{0}{-1} / \triple{u}{0}{0}$. The proof is similar.

Thus, $\textbf{Int} + Ax(\mathcal{M}) \vdash E_{s_{i+1},m_{i+1},n_{i+1}} \to E_{s_i,m_i,n_i}$ for all $i$, $0 \leq i < \mu$. The lemma is proved.
\end{proof}

Since the configuration problem for the Minsky machine $\mathcal{M}$ and the initial configuration $\triple{s_0}{m_0}{n_0}$ is undecidable by Theorem~\ref{T:Minsky}, we have that the derivation problem for the superintuitionistic propositional calculus $\textbf{Int} + Ax(\mathcal{M})$ is also undecidable. This completes the proof of Theorem~\ref{T:main}.

\section{Conclusion and further research}

In this paper, we established that there is an undecidable superintuitionistic propositional calculus using axioms in only 3 variables. Since there are no undecidable superintuitionistic propositional calculi with axioms containing less than 3 variables, therefore a natural and interesting question is there an intuitionistic propositional formula $A$ containing less than 3 variables for which the superintuitionistic propositional calculus $\mathbf{Int} + A$ is undecidable. In this respect, we note that every intermediate logic axiomatised by a 1-variable formula has the finite model property~\cite{Sobolev:77:FASL} and therefore decidable, but there exists an intermediate logic axiomatised by a 2-variable formula, which is Kripke incomplete~\cite{Shehtman:77:IPL}.


\end{document}